\documentclass{article}

\usepackage{color}
 \catcode`@=11
\usepackage{color}
\usepackage{amsthm,amssymb,amsfonts,mathrsfs}
\usepackage{amsmath}
\catcode`@=11 \@addtoreset{equation}{section}

\catcode`@=12

\newtheorem{theorem}{Theorem}[section]
\newtheorem{proposition}{Proposition}[section]
\newtheorem{lemma}{Lemma}[section]

\theoremstyle{definition}

\newtheorem{remark}{Remark}

\def\R{{\mathbb{R}}}

\newcommand{\ds}{\displaystyle}

\usepackage[colorlinks=true,urlcolor=red,linkcolor=blue,citecolor=red,bookmarks=red]{hyperref}
\title{Null controllability for the singular heat equation with a memory term}
\author{{\sc Brahim Allal}
\thanks{The author thanks the MAECI (Ministry of Foreign
Affairs and International Cooperation, Italy) for funding that
greatly facilitated scientific collaboration between Universit\'e Hassan $1^{er}$
(Morocco) and Universit\`{a} di Bari Aldo Moro (Italy).}
\\
Facult\'e des Sciences et Techniques\\
Universit\'e Hassan 1er, 
Laboratoire MISI,\\ B.P. 577, 
Settat 26000, Morocco
\\ email: b.allal@uhp.ac.ma
\\
{\sc Genni Fragnelli}
\thanks{The author is a member of the Gruppo Nazionale per l'Analisi Ma\-te\-matica, la Probabilit\`a e le loro Applicazioni (GNAMPA) of the Istituto Nazionale di Alta Matematica (INdAM) and she is supported by the FFABR {\it Fondo per il finanziamento delle attivit\`a base di ricerca} 2017, by the INdAM - GNAMPA Project 2019 {\it Controllabilit\`a di PDE in modelli fisici e in scienze della vita},  by  Fondi di Ateneo 2015/16 of the University of  Bari {\em Problemi differenziali non linearii} and by PRIN 2017-2019 {\it Qualitative and quantitative aspects of nonlinear PDEs}.}
\\
Dipartimento di Matematica\\ Universit\`{a} di Bari Aldo Moro\\
Via
E. Orabona 4\\ 70125 Bari - Italy\\ email: genni.fragnelli@uniba.it\\
{\sc  Jawad Salhi}\\
Facult\'e des Sciences et Techniques,\\
Universit\'e Hassan 1er, Laboratoire MISI,\\
B.P. 577, Settat 26000, Morocco\\
email: sj.salhi@gmail.com}

\date{}
\begin{document}

\maketitle

\begin{abstract}
In this paper we focus on the null controllability problem for the heat equation with the so-called inverse square potential and a memory term.
To this aim, we first establish the null controllability for a nonhomogeneous singular heat equation by a new Carleman inequality with weights
which do not blow up at $t=0$. Then the null controllability property is proved for the singular heat equation with memory under a condition on the
kernel, by means of Kakutani's fixed-point Theorem.
\end{abstract}



\section{Introduction}
\label{intro}

In this paper, we address the null controllability for the following singular heat equation with memory:
\begin{equation}\label{problem}
\left\{
  \begin{array}{ll}
y_t - y_{xx} - \ds\frac{\mu}{x^2}y  =\int_{0}^{t}a(t,s,x)y(s,x)\,ds + 1_{\omega} u, & (t,x) \in Q,\\
y(t,0)=y(t,1)=0, & t \in (0,T), \\
y(0,x)=y_0(x), & x \in (0,1),
  \end{array}
\right.
\end{equation}
where $y_0 \in L^{2}(0,1)$, $T>0$ is fixed, $\mu$ is a real parameter,
$Q:=(0,T) \times (0,1)$ and $1_{\omega}$ stands for a characteristic function of a nonempty open subset $\omega$ of $(0,1)$.
Here $y$ and $u$ are the state variable and the control variable respectively, $a$ is a given $L^{\infty}$ function defined on
$(0,T)\times Q$.

The analysis of evolution equations involving memory terms is a topic in continuous development.
In the last decades, many researchers have started devoting their attention to this branch of mathematics, motivated by  many applications in  modelling  phenomena in which the processes are affected not only by its current state but also by its history.
Indeed, there is a large spectrum of situations in which the presence of the memory may render the description of the phenomena more accurate.
This is particularly the case for models such as heat conduction in materials with memory, viscoelasticity, theory of population dynamics and nuclear reactors, where there is often a need to reflect the effects of the memory of the system (see for instance \cite{amendola,bloom,lak,pruss}).

Controllability problems for evolution equations with memory terms have been extensively studied in the past. Among other contributions, we mention \cite{barbu, fu, guerrero, halanay, HP2012, ivanov, lavanya, balachandran, TG2016} which, as in our case, deal with parabolic type equations. We also refer to \cite{pandolfi} for an overview of the bibliography on control problems for systems with persistent memory. The first results for a degenerate parabolic equation with memory can be found in \cite{Allal2020}.

In this work, for the first time to our knowledge, we study the null controllability for \eqref{problem}. We underline that here we consider not only a memory term but also  a singular potential one.
In other words, given any $y_0\in L^2(0,1)$, we want to show that there exists a control function $u\in L^2(Q)$
such that the corresponding solution $y$ to \eqref{problem} satisfies
$y(T,x) = 0$ for every $x\in [0,1]$. First results in this direction are obtained in \cite{VanZua2008} in the absence of a memory term when $\mu\leq \ds\frac{1}{4}$ (see also \cite{vz} for the wave and Schr\"{o}dinger equations and \cite{Cazacu2014} for boundary singularity).
Indeed, for the equation
\begin{equation}\label{4}
  u_t -
\Delta u-\mu\frac{1}{|x|^2}u=0, \quad (t,x) \in (0,T)\times \Omega,
\end{equation}
with associated Dirichlet boundary conditions in a bounded domain $\Omega\subset \R^N$ containing the singularity $x=0$ in the interior,  the value of the parameter $\mu$ determines the behavior of the
equation: if $\mu \le {1}/{4}$ (which is the optimal constant of the
Hardy inequality, see \cite{b}) global positive solutions exist, while, if $\mu >{1}/{4}$, instantaneous and complete blow-up occurs
(for other comments on this argument we refer to \cite{Van2011}). In the case of global positive solutions, hence if $\mu\leq\ds \frac{1}{4}$ and using Carleman estimates, in \cite{VanZua2008} it has been proved that such equations can be controlled (in
any time $T > 0$) by a locally distributed control. On the contrary, if $\mu>\ds\frac{1}{4}$, the null controllability fails as shown in \cite{Ervedoza2008}. After these first results, several other works followed extending them in various situations (see for instance \cite{biccari2019, BZ2016, Cazacu2014, F2016, FMpress, FM2020, FM2019, FM2018, FM2017, MarVan2019,Van2011}).

However, when $\mu=0$ and $a=1$, \eqref{problem} becomes the following control system
associated to the classical heat equation with memory:
\begin{equation}\label{problem0}
\left\{
  \begin{array}{ll}
y_t - y_{xx}  =\int_{0}^{t}y(s)\,ds + 1_{\omega} u, & (t,x) \in Q,\\
y(t,0)=y(t,1)=0, & t \in (0,T), \\
y(0,x)=y_0(x), & x \in (0,1).
  \end{array}
\right.
\end{equation}
In this case, as shown in \cite{guerrero,zhou}, there exists a set of initial conditions such that the null controllability property for
 \eqref{problem0} fails whenever the control region $\omega$ is fixed, independent of time. For some related works in this respect we also refer to \cite{chaves,HP2012,zhou2016}.

Nevertheless, since the positive controllability results are important in real world applications, it is natural to analyze whether it is possible that control properties for \eqref{problem} could be obtained. For this reason, under suitable conditions on the singularity parameter $\mu$ and on the kernel $a$, we establish that \eqref{problem} is null controllable.

Our approach is inspired from the techniques presented in the work \cite{TG2016} for
the Laplace operator, suitably adapted in order to deal with the additional inverse-square potential. In particular, the technique that we will use is based on appropriate Carleman estimates and on the fixed-point Theorem of Kakutani.

The paper is organized as follows: Section \ref{NHE} is devoted to the study of null controllability for a nonhomogeneous singular heat equation without memory via new Carleman estimates. In Section \ref{SHEM}, the null controllability for the singular heat equation with memory \eqref{problem} is proved.

A final comment on the notation: by $C$ we shall denote
universal positive constants, which are allowed to vary from line to
line.

\section{Nonhomogeneous singular heat equation}
\label{NHE}
In this section, we prove the null controllability for a nonhomogeneous singular heat equation using a new modified Carleman inequality. This null controllability result is the key tool for the controllability of the heat equation with memory. Thus, as a first step, we consider the following problem:
\begin{equation}\label{problem01}
\left\{
  \begin{array}{ll}
y_t - y_{xx} - \ds\frac{\mu}{x^2}y  =f + 1_{\omega} u(t), & (t,x) \in Q:=(0,T)\times(0,1),\\
y(t,0)=y(t,1)=0, & t \in (0,T), \\
y(0,x)=y_0(x), & x \in (0,1),
  \end{array}
\right.
\end{equation}
where $f\in L^2(Q)$ is a given source term.

Prior to null controllability is the well-posedness of \eqref{problem01}, a question
we address in the next subsection.
\subsection{Functional framework and well-posedness}
\label{CP}
We analyze here existence and uniqueness of solutions for the
heat problem \eqref{problem01}. To simplify the presentation, we first focus on the
well-posedness of the following inhomogeneous singular problem
\begin{equation}\label{problem1}
\left\{
  \begin{array}{ll}
y_t - y_{xx} -\ds \frac{\mu}{x^2}y  =f, & (t,x) \in Q,\\
y(t,0)=y(t,1)=0, & t \in (0,T), \\
y(0,x)=y_0(x), & x \in (0,1).
  \end{array}
\right.
\end{equation}
In this framework, in order to deal with the singularity of the potential, a fundamental tool is the very famous Hardy inequality.
To fix the ideas, we recall here the basic form of the Hardy inequality in dimension one
(see, for example, \cite[Theorem 327]{HLP1952} or \cite[Lemma 5.3.1]{Davies1995}):
\begin{equation}\label{hardyinequality}
\frac{1}{4}\int_{0}^{1}\frac{y^2}{x^2}\,dx \leq \int_{0}^{1} y^2_x\,dx,
\end{equation}
which is valid for every $ y\in H^1(0,1)$ with $y(0)=0$.

Now, for any $\mu\leq \ds\frac{1}{4}$, we define
\begin{align*}
H_{0}^{1,\mu}(0,1):= \Big\{ y \in L^2(0,1)\cap H^{1}_{loc}((0,1])\,&\mid z(0)=z(1)=0 ,\\
&\text{and}\quad \int_{0}^{1}(z_x^2 - \mu \frac{z^2}{x^2})\,dx<+\infty \Big\}.
\end{align*}
Note that $H_{0}^{1,\mu}(0,1)$ is a Hilbert space
obtained as the completion of $C^{\infty}_{c}(0,1)$, or $H^1_0(0,1)$, with respect to the norm
\begin{equation*}
\|y\|_{\mu}:= \Big(\int_{0}^{1}(y_x^2 - \mu \frac{y^2}{x^2})\,dx\Big)^{\frac{1}{2}}, \qquad \forall \; y\in H^1_0(0,1).
\end{equation*}
In the case of a sub-critical parameter $\mu < \ds\frac{1}{4}$, thanks to the Hardy inequality \eqref{hardyinequality},
one can see that $\|\cdot\|_\mu$ is equivalent to the standard norm of $H^1_0(0,1)$, and thus $H_{0}^{1,\mu}(0,1)=H_{0}^{1}(0,1)$.
In the critical case $\mu=\ds\frac{1}{4}$, it is proved
(see \cite{VazZua2000}) that this identification does not hold anymore and the space $H_{0}^{1,\mu}(0,1)$ is slightly (but strictly) larger than
$H_{0}^{1}(0,1)$.

Now, define the operator $A : D(A)\subset L^2(0,1) \rightarrow L^2(0,1)$ corresponding to the heat equation
with an inverse square potential in the following way:
\[Ay:= -y_{xx} - \ds\frac{\mu}{x^2}y\]
\[\forall \;y \in  D(A):= \left\{y \in H^2_{loc}((0,1])\cap H_{0}^{1,\mu}(0,1): y_{xx} +\ds \frac{\mu}{x^2}y \in L^2(0,1)\right\}.\]

In this context, $A$ is self-adjoint, nonpositive on $L^2(0,1)$ and it generates an analytic semi-group of contractions in $L^2(0,1)$ for the
equation \eqref{problem1} (see \cite{VazZua2000}). Consequently, the singular heat equation \eqref{problem1} is well-posed.
To be precise, the next result holds.
\begin{theorem}\label{prop} For all $f\in L^2(Q)$ and $y_0 \in L^2(0,1)$, there exists a unique solution
\[y \in \mathcal{W}:=C\big([0,T]; L^2(0,1)\big) \cap L^2 \big(0,T;H_{0}^{1,\mu}(0,1)\big)\]
of \eqref{problem1} such that
\begin{equation}\label{stima2w}
\sup_{t \in [0,T]}
\|y(t)\|^2_{L^2(0,1)}+\int_0^T\|y(t)\|^2_{\mu} dt \leq
C_T\left(\|y_0\|^2_{L^2(0,1)}+\|f\|^2_{L^2(Q)}\right),
\end{equation}
for some positive constant $C_T$. Moreover, if $y_0 \in H_0^{1,\mu}(0,1)$, then
\begin{equation}\label{regularity1}
y\in \mathcal{Z}:=H^1\big(0,T; L^2(0,1)\big) \cap L^2\big(0,T;D(A)\big)\cap C\big([0,T];H_{0}^{1,\mu}(0,1)\big),
\end{equation}
and there exists a positive constant $C$ such that
\begin{equation}\label{stima3w}
\begin{aligned}
\sup_{t \in [0,T]}\left(\|y(t)\|^2_{\mu}
\right)+ \int_0^{T}
\left(\left\|y_t\right\|^2_{L^2(0,1)} +
\left\|y_{xx} + \frac{\mu}{x^2}y\right\|^2_{L^2(0,1)}\right)dt
\leq C
\left(\|y_0\|^2_{\mu} +
\|f\|^2_{L^2(Q)}\right).
\end{aligned}
\end{equation}
\end{theorem}
\begin{proof}
In \cite{VazZua2000}, the authors use semigroup theory to obtain the well-posedness result for the problem \eqref{problem1} (see also \cite{MarVan2019}).
Thus, in the rest of the proof, we will prove only \eqref{stima2w}-\eqref{stima3w}. First, being
$A$ the generator of a strongly continuous semigroup on
$L^2(0,1)$, if $y_0\in L^2(0,1)$, then the solution $y$ of \eqref{problem1} belongs to $C\big([0,T];L^2(0,1)\big) \cap L^2 \big(0,T;H^{1,\mu}_{0}(0,1)\big)$, while,
if $y_0\in D(A)$, then $y\in H^1\big(0,T; L^2(0,1)\big) \cap L^2\big(0,T;D(A)\big)$.

Now, by a usual energy method we shall prove \eqref{regularity1} and \eqref{stima3w}, from which the last required
regularity property for $y$ will follow by standard linear arguments.
First, take $y_0\in D(A)$ and multiply the equation of \eqref{problem1} by $y$.
By the Cauchy-Schwarz inequality we obtain for every $t\in (0,T]$,
\begin{equation}\label{derivo}
\frac{1}{2}\frac{d}{dt}\|y(t)\|^2_{L^2(0,1)}+
\|y(t)\|^2_{\mu}\leq
\frac{1}{2}\|f(t)\|^2_{L^2(0,1)}+\frac{1}{2}
\|y(t)\|^2_{L^2(0,1)}.
\end{equation}
From \eqref{derivo} and using Gronwall's inequality, we get
\begin{equation}\label{sottoderivo}
\|y(t)\|^2_{L^2(0,1)}\leq
e^T\left(\|y(0)\|^2_{L^2(0,1)}+\|f\|_{L^2(Q)}^2
\right)
\end{equation}
for every $t\leq T$. From
\eqref{derivo} and \eqref{sottoderivo} we immediately obtain
\begin{equation}\label{sottosotto}
\int_0^T\|y(t)\|^2_{\mu}dt\leq
C_T\left(\|y(0)\|^2_{L^2(0,1)}+\|f\|_{L^2(Q)}^2
\right)
\end{equation}
for some universal constant $C_T>0$. Thus, by
\eqref{sottoderivo} and \eqref{sottosotto}, \eqref{stima2w} follows
if $y_0\in D(A)$. Since $D(A)$ is dense in $L^2(0,1)$ (see \cite{Vanb,VazZua2000}),
the same inequality holds if $y_0\in L^2(0,1)$.

Now, multipling the equation by $\ds-y_{xx} - \frac{\mu}{x^2}y$, integrating on $(0,1)$ and using the Cauchy-Schwarz inequality,
we easily get
$$
\displaystyle \frac{d}{dt}\|y(t)\|^2_{\mu}+\|y_{xx}(t) + \frac{\mu}{x^2}y(t)\|^2_{L^2(0,1)}\leq
\|f(t)\|_{L^2(0,1)}^2
$$
for every $t\in [0,T]$, so that, as before, we find $C_T'>0$ such that
\begin{equation}\label{mah}
\|y(t)\|^2_{\mu}+\int_0^T\|y_{xx}(t) + \frac{\mu}{x^2}y(t)\|^2_{L^2(0,1)}dt
\leq
C_T'\left(\|y(0)\|_{\mu}+\|f\|_{L^2(Q)}^2\right)
\end{equation}
for every $t\!\leq \!T$.\!
Finally, from $y_t=y_{xx}+\ds\frac{\mu}{x^2}y+f$, squaring and integrating on $Q$,
we find
$$
\!\int_0^T\!\!\!\|y_t(t)\|_{L^2(0,1)}^2\!\leq\!
C\left(\!\int_0^T\|y_{xx}+\frac{\mu}{x^2}y\|^2_{L^2(0,1)}\!\!+\!\!\|f\|_{L^2(Q)}^2
\!\right),
$$
and together with \eqref{mah} we find
\begin{equation}\label{allafine}
\int_0^T\|y_t(t)\|_{L^2(0,1)}^2\leq
C\left(\|y(0)\|^2_{\mu}+\|f\|_{L^2(Q)}^2\right).
\end{equation}

In conclusion, \eqref{derivo}, \eqref{sottoderivo}, \eqref{mah} and \eqref{allafine}
give \eqref{stima2w} and \eqref{stima3w}. Notice that, \eqref{regularity1} and \eqref{stima3w}
hold also if $y_0 \in H_0^{1,\mu}(0,1)$.
\end{proof}

\subsection{Carleman estimates for a singular problem}
\label{CESHE}
In this subsection we prove a new Carleman estimate for the adjoint parabolic equation associated to \eqref{problem01}, which will provide that the nonhomogeneous singular heat equation \eqref{problem01} is null controllable.
Hence, in the following, we concentrate on the
next adjoint problem
\begin{equation}\label{adjproblem}
\left\{
  \begin{array}{ll}
-z_t - z_{xx} - \ds\frac{\mu}{x^2}z  =g, & (t,x) \in Q,\\
z(t,0)=z(t,1)=0, & t \in (0,T), \\
z(T,x)=z_T(x), & x \in (0,1).
  \end{array}
\right.
\end{equation}
Following \cite{VanZua2008}, for every $0<\gamma<2$, let us introduce the weight
function
\begin{equation}\label{phi}
\varphi(t,x):=\theta(t)\psi(x),
\end{equation}
where
\begin{equation}\label{weightfunc}
\psi(x):= \mathfrak{c}(x^2 - d), \quad \theta(t):=\left(\frac{1}{t(T-t)}\right)^k, \quad  k:=1+\frac{2}{\gamma},
\end{equation}
$\mathfrak{c}>0$ and $d>1$. A more precise restriction on the parameters $k, \mathfrak{c}$ and $d$  will be needed later.
Observe that
$\displaystyle\lim_{t\rightarrow 0^+}\theta(t)=\displaystyle\lim_{t\rightarrow T^-}\theta(t)= + \infty,$ and
$$
\psi(x)<0\qquad \text{for every}\quad x\in[0,1].
$$

Using the previous weight functions and the following  improved Hardy-Poincar\'e inequality given in \cite{Van2011}:

{\it
For all $\eta > 0$, there exists some positive
constant $C = C(\eta) > 0$ such that, for all $z\in C^{\infty}_c(0,1):$
\begin{equation}\label{ihpi}
\int_{0}^{1} x^{\eta} z_x^2\,dx \leq C\int_{0}^{1}\left(z_x^2 - \frac{1}{4}\frac{z^2}{x^2}\right)\,dx,
\end{equation}}
one can prove the following Carleman estimate for the case of a purely singular parabolic equation:

\begin{lemma}\label{Carleman001}\cite[Theorem 5.1]{Van2011}
Assume that $\mu\leq \ds\frac{1}{4}$. Then, there exists
$C > 0$ and $s_0> 0$ such that, for all $s\geq s_0$, every solution $z$ of \eqref{adjproblem}
satisfies
\begin{align}\label{Carl001}
&\int\!\!\!\!\!\int_{Q} s^{3} \theta^3 x^2 z^{2} e^{2s\varphi}\,dx\,dt
+\int\!\!\!\!\!\int_{Q} s \theta \big(z_x^2- \mu \frac{z^{2}}{x^2} \big)e^{2s\varphi}\,dx\,dt
+\int\!\!\!\!\!\int_{Q} s \theta  \frac{z^{2}}{x^\gamma} e^{2s\varphi}\,dx\,dt \nonumber\\
&\leq C\Big(\int\!\!\!\!\!\int_{Q} g^2 e^{2s\varphi}\,dx\,dt + \int_{0}^{T}s\theta z_x^2(t,1)e^{2s\varphi(t,1)}\,dx\,dt\Big).
\end{align}
\end{lemma}

Observe that, if the term
$$
\int\!\!\!\!\!\int_{Q} s \theta \big(z_x^2- \mu \frac{z^{2}}{x^2} \big)e^{2s\varphi}\,dx\,dt
$$
is not positive, then the estimate \eqref{Carl001} is not of great importance.
In fact, the Hardy inequality \eqref{hardyinequality} only ensures the positivity of of the quantity
$$
\int\!\!\!\!\!\int_{Q} s \theta \big(z_x^2- \mu \frac{z^{2}}{x^2} \big)\,dx\,dt.
$$
However, from \cite[Remark 3]{Van2011} and similarly
as in \cite{hajjaj}, we will rewrite the result given in Lemma \ref{Carleman001} in a more practical way.

\begin{lemma}\label{Carleman1}
Assume that $\mu\leq \ds\frac{1}{4}$. Then, there exist
$C > 0$ and $s_0> 0$ such that, for all $s\geq s_0$, every solution $z$ of \eqref{adjproblem}
satisfies
\begin{equation}\label{Carl10}
\mathfrak{J}_{\varphi,\eta,\gamma}(z) \leq C\Big(\int\!\!\!\!\!\int_{Q} g^2 e^{2s\varphi}\,dx\,dt + \int_{0}^{T}s\theta z_x^2(t,1)e^{2s\varphi(t,1)}\,dx\,dt\Big),
\end{equation}
where
\begin{align}\label{carloper1}
\mathfrak{J}_{\varphi,\eta,\gamma}(z)&=\int\!\!\!\!\!\int_{Q} s^{3} \theta^3 x^2 z^{2} e^{2s\varphi}\,dx\,dt
+\int\!\!\!\!\!\int_{Q} s \theta z_x^2 e^{2s\varphi}\,dx\,dt \nonumber\\
&\quad+ \int\!\!\!\!\!\int_{Q} s \theta  \frac{z^{2}}{x^2} e^{2s\varphi}\,dx\,dt
+\int\!\!\!\!\!\int_{Q} s \theta  \frac{z^{2}}{x^\gamma} e^{2s\varphi}\,dx\,dt,
\end{align}
if $\mu <\ds\frac{1}{4}$, and
\begin{align}\label{carloper2}
\mathfrak{J}_{\varphi,\eta,\gamma}(z)&=\int\!\!\!\!\!\int_{Q} s^{3} \theta^3 x^2 z^{2} e^{2s\varphi}\,dx\,dt
+\int\!\!\!\!\!\int_{Q} s \theta x^{\eta} z_x^2 e^{2s\varphi}\,dx\,dt \notag\\
&\quad+\int\!\!\!\!\!\int_{Q} s \theta  \frac{z^{2}}{x^\gamma} e^{2s\varphi}\,dx\,dt,
\end{align}
if $\mu =\ds\frac{1}{4}$. Here $\gamma$ is as in \eqref{weightfunc}
\end{lemma}
\begin{proof}

\noindent\textbf{Case 1: If $\mu<\ds\frac{1}{4}$.}\\
Let $Z=z e^{s\varphi}$. In order to prove \cite[Theorem 5.1]{Van2011}, the author has derived the following estimate
\begin{align}\label{step1}
&\int\!\!\!\!\!\int_{Q} s^{3} \theta^3 x^2 Z^{2}\,dx\,dt
+\int\!\!\!\!\!\int_{Q} s \theta \big(Z_x^2- \mu \frac{Z^{2}}{x^2} \big)\,dx\,dt
+\int\!\!\!\!\!\int_{Q} s \theta  \frac{Z^{2}}{x^\gamma} \,dx\,dt \nonumber\\
&\leq C\Big(\int\!\!\!\!\!\int_{Q} g^2 e^{2s\varphi}\,dx\,dt + \int_{0}^{T}s\theta Z_x^2(t,1)\,dx\,dt\Big).
\end{align}
Let $\delta< \inf(1,(1-4\mu))$ be a fixed positive constant. We have
\begin{align}\label{step2}
\int\!\!\!\!\!\int_{Q} s \theta \big(Z_x^2- \mu \frac{Z^{2}}{x^2} \big)\,dx\,dt&= (1-\delta)\int\!\!\!\!\!\int_{Q} s \theta \big(Z_x^2- \frac{1}{4} \frac{Z^{2}}{x^2} \big)\,dx\,dt \nonumber\\
&\quad+\delta \int\!\!\!\!\!\int_{Q} s \theta Z_x^2\,dx\,dt + \left(\frac{1}{4}(1-\delta)-\mu\right)\int\!\!\!\!\!\int_{Q} s \theta \frac{Z^{2}}{x^2}\,dx\,dt.
\end{align}
By \eqref{step1} and \eqref{step2}, we obtain
\begin{align*}
&\int\!\!\!\!\!\int_{Q} s^{3} \theta^3 x^2 Z^{2} \,dx\,dt
+ (1-\delta)\int\!\!\!\!\!\int_{Q} s \theta \big(Z_x^2- \frac{1}{4} \frac{Z^{2}}{x^2} \big)\,dx\,dt
+\delta \int\!\!\!\!\!\int_{Q} s \theta Z_x^2\,dx\,dt\\
&+ \left(\frac{1}{4}(1-\delta)-\mu\right)\int\!\!\!\!\!\int_{Q} s \theta \frac{Z^{2}}{x^2}\,dx\,dt
+\int\!\!\!\!\!\int_{Q} s \theta  \frac{Z^{2}}{x^\gamma} \,dx\,dt \\
&\leq C\Big(\int\!\!\!\!\!\int_{Q} g^2 e^{2s\varphi}\,dx\,dt + \int_{0}^{T}s\theta Z_x^2(t,1)\,dx\,dt\Big).
\end{align*}
On the other hand, from \eqref{ihpi}, for all $\eta>0$ there exists a constant $c_0=c_0(\eta)>0$ such that
\begin{equation}\label{step3}
\int\!\!\!\!\!\int_{Q} s \theta \big(Z_x^2- \frac{1}{4} \frac{Z^{2}}{x^2} \big)\,dx\,dt \geq c_0 \int\!\!\!\!\!\int_{Q} s \theta x^{\eta} Z_x^2\,dx\,dt.
\end{equation}
Hence,
\begin{align}\label{step4}
&\int\!\!\!\!\!\int_{Q} s^{3} \theta^3 x^2 Z^{2} \,dx\,dt
+ (1-\delta)c_0 \int\!\!\!\!\!\int_{Q} s \theta x^{\eta} Z_x^2\,dx\,dt
+\delta \int\!\!\!\!\!\int_{Q} s \theta Z_x^2\,dx\,dt \notag\\
&+ \big(\frac{1}{4}(1-\delta)-\mu\big)\int\!\!\!\!\!\int_{Q} s \theta \frac{Z^{2}}{x^2}\,dx\,dt
+\int\!\!\!\!\!\int_{Q} s \theta  \frac{Z^{2}}{x^\gamma} \,dx\,dt \notag\\
&\leq C\Big(\int\!\!\!\!\!\int_{Q} g^2 e^{2s\varphi}\,dx\,dt + \int_{0}^{T}s\theta Z_x^2(t,1)\,dx\,dt\Big).
\end{align}
Using the definition of $Z$, we have
\begin{equation}\label{step50}
Z^2= z^2 e^{2s\varphi},
\end{equation}
\begin{align}\label{step51}
Z_x= z_xe^{s\varphi} +s\theta \psi_x Z \quad \text{and}\quad z_x^2 e^{2s\varphi} \leq 2 Z_x^2 + c s^2\theta^2 x^2 Z^2,
\end{align}
for  a positive constant $c$.
Then,
\begin{equation}\label{step6}
\int\!\!\!\!\!\int_{Q} s \theta z_x^2e^{2s\varphi} \,dx\,dt \leq 2 \int\!\!\!\!\!\int_{Q} s \theta Z_x^2 \,dx\,dt
+ c \int\!\!\!\!\!\int_{Q} s^3 \theta^3 x^2 Z^2 \,dx\,dt.
\end{equation}
Combining \eqref{step4}-\eqref{step6}, we obtain the desired estimate \eqref{Carl10}. Indeed, defining
\[
a_0=\min\left\{\frac{1}{1+c},\frac{\delta}{2}, \left(\frac{1}{4}(1-\delta)-\mu\right)\right\}>0,
\]
we have
\begin{equation}
\begin{aligned}
&a_0\left(\int\!\!\!\!\!\int_{Q} s^{3} \theta^3 x^2 z^{2} e^{2s\varphi}\,dx\,dt
+\int\!\!\!\!\!\int_{Q} s \theta z_x^2 e^{2s\varphi}\,dx\,dt \nonumber+ \int\!\!\!\!\!\int_{Q} s \theta  \frac{z^{2}}{x^2} e^{2s\varphi}\,dx\,dt
+\int\!\!\!\!\!\int_{Q} s \theta  \frac{z^{2}}{x^\gamma} e^{2s\varphi}\,dx\,dt\right)\\
&\le a_0\left( (1+c)\int\!\!\!\!\!\int_{Q} s^{3} \theta^3 x^2Z^2\,dx\,dt + 2\int\!\!\!\!\!\int_{Q} s\theta Z_x^2dxdt +\int\!\!\!\!\!\int_{Q} s \theta  \frac{Z^{2}}{x^2} \,dx\,dt
+\int\!\!\!\!\!\int_{Q} s \theta  \frac{Z^{2}}{x^\gamma}\,dx\,dt\right)\\
& \le \int\!\!\!\!\!\int_{Q} s^{3} \theta^3 x^2Z^2\,dx\,dt + \delta \int\!\!\!\!\!\int_{Q} s\theta Z_x^2dxdt +\left(\frac{1}{4}(1-\delta)-\mu\right)\int\!\!\!\!\!\int_{Q} s \theta  \frac{Z^{2}}{x^2} \,dx\,dt
+\int\!\!\!\!\!\int_{Q} s \theta  \frac{Z^{2}}{x^\gamma}\,dx\,dt\\
&\le \int\!\!\!\!\!\int_{Q} s^{3} \theta^3 x^2 Z^{2} \,dx\,dt
+ (1-\delta)c_0 \int\!\!\!\!\!\int_{Q} s \theta x^{\eta} Z_x^2\,dx\,dt
+\delta \int\!\!\!\!\!\int_{Q} s \theta Z_x^2\,dx\,dt \notag\\
&+ \left(\frac{1}{4}(1-\delta)-\mu\right)\int\!\!\!\!\!\int_{Q} s \theta \frac{Z^{2}}{x^2}\,dx\,dt
+\int\!\!\!\!\!\int_{Q} s \theta  \frac{Z^{2}}{x^\gamma} \,dx\,dt \notag\\
&\leq C\Big(\int\!\!\!\!\!\int_{Q} g^2 e^{2s\varphi}\,dx\,dt + \int_{0}^{T}s\theta Z_x^2(t,1)\,dx\,dt\Big).
\end{aligned}
\end{equation}
Thus, the conclusion follows.

\noindent\textbf{Case 2: If $\mu=\ds\frac{1}{4}$.}\\
As before, let $Z=z e^{s\varphi}$ and define
\[
a_0=\min\left\{\frac{1}{1+c},\frac{c_0}{2} \right\}>0,
\]
where $c_0$ and $c$ are the constants of \eqref{step3} and \eqref{step51}, respectively. Then,
by \eqref{step1}, \eqref{step3}, \eqref{step50} and \eqref{step51}, that still hold if $\mu=\ds\frac{1}{4}$, we have
\begin{equation}\label{step7}
\begin{aligned}
&a_0\left(\int\!\!\!\!\!\int_{Q} s^{3} \theta^3 x^2 z^{2} e^{2s\varphi}\,dx\,dt
+ \int\!\!\!\!\!\int_{Q} s \theta x^{\eta} z_x^2e^{2s\varphi}\,dx\,dt
+\int\!\!\!\!\!\int_{Q} s \theta  \frac{z^{2}}{x^\gamma} e^{2s\varphi}\,dx\,dt\right) \\
&\le a_0\left(\int\!\!\!\!\!\int_{Q} s^{3} \theta^3 x^2 Z^2\,dx\,dt+  2\int\!\!\!\!\!\int_{Q} s \theta x^{\eta} Z_x^2 \,dx\,dt
+ c \int\!\!\!\!\!\int_{Q} s^3 \theta^3 x^2 Z^2  \,dx\,dt+ \int\!\!\!\!\!\int_{Q} s \theta  \frac{Z^{2}}{x^\gamma} \,dx\,dt\right)\\
& \le a_0(1+c) \int\!\!\!\!\!\int_{Q} s^{3} \theta^3 x^2 Z^2\,dx\,dt+a_0  \frac{2}{c_0}\int\!\!\!\!\!\int_{Q} s \theta \big(Z_x^2- \frac{1}{4} \frac{Z^{2}}{x^2} \big)\,dx\,dt+a_0 \int\!\!\!\!\!\int_{Q} s \theta  \frac{Z^{2}}{x^\gamma} \,dx\,dt \\
&(\text{by \eqref{step1}})\\
&\leq C\Big(\int\!\!\!\!\!\int_{Q} g^2 e^{2s\varphi}\,dx\,dt + \int_{0}^{T}s\theta z_x^2(t,1)e^{2s\varphi(t,1)}\,dx\,dt\Big).
\end{aligned}
\end{equation}
Hence, also in this case the conclusion follows.

\end{proof}
We point out that the Carleman estimates stated above are not appropriate to achieve our goal. In fact, all these estimates does not have the observation term in the interior of the domain. However, we use them to obtain the main Carleman estimate stated in Proposition \ref{maincarleman}.
More precisely, from the boundary Carleman estimates \eqref{Carl10}, we will deduce a global Carleman estimate for the adjoint problem \eqref{adjproblem}
with a distributed observation on a subregion
\begin{equation}\label{omegap}
\omega':= (\alpha', \beta') \subset\subset \omega.
\end{equation}
To do so, we recall the following weight functions associated to nonsingular Carleman estimates which are suited to our purpose:
$$\Phi(t,x):=\theta(t)\Psi(x)$$
where $\theta$ is defined in \eqref{weightfunc} and $\Psi(x) = e^{\rho \sigma} - e^{2 \rho \|\sigma\|_{\infty}}$. Here $\rho>0$, $\sigma \in C^2([0,1])$ is such that $\sigma(x)>0$ in $(0,1)$, $\sigma(0)=\sigma(1)=0$ and $\sigma_x (x)\neq 0$ in $[0,1]\setminus\tilde{\omega}$, being $\tilde{\omega}$ an arbitrary open subset of $\omega$.

In the following, we choose the constant $\mathfrak{c}$ in \eqref{weightfunc} so that
$$
\mathfrak{c}\geq \frac{e^{2 \rho \|\sigma\|_{\infty}} -1}{d-1}.
$$
By this choice one can prove that the function $\varphi$ defined in \eqref{phi} satisfies the next estimate
\begin{equation}\label{star10}
\varphi(t,x)\leq \Phi(t,x)\qquad \text{for every}\quad (t,x)\in [0,T]\times[0,1].
\end{equation}

Thanks to this property, we can prove the main Carleman estimate of this paper whose proof is based also on the following Caccioppoli's inequality:
\begin{proposition}[Caccioppoli's inequality]\label{lem_Caccio}
Let $\omega'$ and $\omega''$ be two nonempty open subsets of $(0, 1) $ such that $ \overline{\omega''} \subset \omega'$ and $\phi(t,x) = \theta(t) \varrho(x)$, where $\varrho \in C^2(\overline{\omega'}, \, \mathbb{R})$.  Then, there exists a constant $C> 0$ such that any solution $z$ of \eqref{adjproblem} satisfies
\begin{equation}\label{Caccioppoli_ineq}
\int\!\!\!\!\!\int_{Q_{\omega''}}  z_x^{2} e^{2s\phi}\,dx\,dt  \leq C
\int\!\!\!\!\!\int_{Q_{\omega'}} (g^{2} +  s^2 \theta^2  z^{2})e^{2s\phi}\,dx\,dt,
\end{equation}
where $Q_{\omega}:=(0,T)\times\omega$.
\end{proposition}

The proof of the previous result is similar to the one given, for instance, in \cite[Lemma 6.1]{allal}, so we omit it.

Now, we are ready to prove the following result:
\begin{proposition}\label{maincarleman}
Assume that $\mu\leq\ds \frac{1}{4}$. Then, there exist two positive constants $C$ and $s_0$ such that, the solution $z$ of equation \eqref{adjproblem} satisfies, for all $s \geq s_0$
\begin{equation}\label{maincarl01}
\mathfrak{J}_{\varphi,\eta,\gamma}(z)
\leq C \Big(\int\!\!\!\!\!\int_{Q} g^2 e^{2s\Phi}\,dx\,dt + \int\!\!\!\!\!\int_{Q_{\omega'}} s^{3}\theta^3 z^2 e^{2s\Phi}\,dx\,dt\Big).
\end{equation}
Here $\mathfrak{J}_{\varphi,\eta,\gamma}(\cdot)$ is defined in \eqref{carloper1} or \eqref{carloper2}.
\end{proposition}
\begin{proof}
Let us set
$\omega'' = (\alpha'', \beta'') \subset\subset \omega'$
and consider a smooth cut-off function $\xi \in C^\infty([0, 1])$ such that $0\leq \xi(x) \leq 1$ for $x\in (0,1)$, $\xi(x)=1$ for $x \in [0, \alpha'']$ and $\xi(x)=0$ for $x\in[\beta'', 1]$. Define $w:= \xi z$ where $z$ is the solution of \eqref{adjproblem}. Then,  $w$ satisfies the following problem:
\begin{equation}\label{equationxi}
\left\{
\begin{array}{lll}
-w_t - w_{xx} - \ds\frac{\mu}{x^2}w = \xi g - \xi_{xx}z - 2\xi_{x} z_x, & & (t,x)\in Q,\\
w(t, 1)=w(t, 0) = 0, & & t \in (0, T),\\
w(T,x)= \xi(x) z_{T}(x), & &  x \in  (0,1).
\end{array}
\right.
\end{equation}
First of all, we prove the first intermediate Carleman estimate for $z$ in $(0,T) \times(0,\alpha')$ (recall that $z\equiv w$ in $[0, \alpha']$):
\begin{equation}\label{intercarl01}
\begin{aligned}
\mathfrak{J}_{\varphi,\eta,\gamma}(w)&
\leq C \Big(\int\!\!\!\!\!\int_{Q} \xi^2 g^2 e^{2s\varphi}\,dx\,dt + \int\!\!\!\!\!\int_{Q_{\omega'}}(g^2 + s^{2}\theta^2 z^2) e^{2s\varphi}\,dx\,dt\Big)\\
& \leq C \Big(\int\!\!\!\!\!\int_{Q} \xi^2 g^2 e^{2s\Phi}\,dx\,dt + \int\!\!\!\!\!\int_{Q_{\omega'}}(g^2 + s^{2}\theta^2 z^2) e^{2s\Phi}\,dx\,dt\Big).
\end{aligned}
\end{equation}
The second inequality in \eqref{intercarl01} follows by \eqref{star10}, thus it is sufficient to prove the first inequality of \eqref{intercarl01}.
Applying the Carleman estimate \eqref{Carl10} to \eqref{equationxi}, we obtain
\begin{equation}\label{star6}
\mathfrak{J}_{\varphi,\eta,\gamma}(w)
\leq C\int\!\!\!\!\!\int_{Q} \Big(\xi^2 g^2 + \big(\xi_{xx}z + 2\xi_{x} z_x\big)^2\Big)e^{2s\varphi}\,dx\,dt.
\end{equation}
From the definition of $\xi$ and the Caccioppoli inequality \eqref{Caccioppoli_ineq}, we obtain
\begin{align}\label{star7}
\int\!\!\!\!\!\int_{Q}\big(\xi_{xx}z + 2\xi_{x} z_x\big)^2 e^{2s\varphi}\,dx\,dt
&\leq C\int\!\!\!\!\!\int_{Q_{\omega''}} (z^2 + z_x^2 )e^{2s\varphi}\,dx\,dt \notag\\
&\leq C \int\!\!\!\!\!\int_{Q_{\omega'}} (g^2 + s^{2}\theta^2 z^2) e^{2s\varphi}\,dx\,dt.
\end{align}
Combining \eqref{star6} and \eqref{star7} we obtain \eqref{intercarl01}.

Now, using the non degenerate Carleman estimate of
\cite[Lemma 1.2]{FI1996}, we are going to show a second estimate of $z$ in $(0,T)\times(\beta',1)$.
For this purpose, let $v=\zeta z$ where $\zeta:=1-\xi$ (hence $z \equiv v$ in $[\beta',1]$).
Clearly, the function $v$ is a solution of the uniformly parabolic equation
\begin{equation}\label{equationzeta}
\left\{
\begin{array}{lll}
-v_t - v_{xx} - \ds\frac{\mu}{x^2}v= \zeta g - \zeta_{xx}z - 2\zeta_{x} z_x, & & (t,x)\in (0,T) \times (\alpha',1),\\
v(t, 1)= v(t, \alpha') = 0, & & t \in (0, T),\\
v(T,x)= \zeta(x) z_{T}(x), & &  x \in  (\alpha',1).
\end{array}
\right.
\end{equation}
Since $\zeta$ has its support in $ [\alpha'', \beta'']$,
by \cite[Lemma 1.2]{FI1996} we have
\begin{align*}
&\int\!\!\!\!\!\int_{Q} \Big( s\theta v_{x}^{2} + s^{3}\theta^3 v^{2}\Big) e^{2s\Phi}\,dx\,dt= \int_0^T\int_{\alpha'}^1 \Big( s\theta v_{x}^{2} + s^{3}\theta^3 v^{2}\Big) e^{2s\Phi}\,dx\,dt\\
&\leq C\Bigg(\int_0^T\int_{\alpha'}^1  \Big(\zeta^2 g^2 + \big(\zeta_{xx}z + 2\zeta_{x} z_x\big)^2\Big)e^{2s\Phi}\,dx\,dt + \int\!\!\!\!\!\int_{Q_{\omega''}} s^{3}\theta^3 v^2 e^{2s\Phi}\,dx\,dt\Bigg)\\
&\leq C\Bigg(\int\!\!\!\!\!\int_{Q} \zeta^2 g^2 e^{2s\Phi}\,dx\,dt + \int\!\!\!\!\!\int_{Q_{\omega''}} (z^2 + z^2_x)e^{2s\Phi}\,dx\,dt + \int\!\!\!\!\!\int_{Q_{\omega''}} s^{3}\theta^3 v^2 e^{2s\Phi}\,dx\,dt\Bigg).
\end{align*}

Therefore, by the previous estimate, by \eqref{star10} and using the Caccioppoli inequality \eqref{Caccioppoli_ineq}, we
deduce
\begin{equation}\label{star9}
\begin{aligned}
&\int\!\!\!\!\!\int_{Q} \Big( s\theta v_{x}^{2} + s^{3}\theta^3 v^{2}\Big) e^{2s\varphi}\,dx\,dt\le \int\!\!\!\!\!\int_{Q} \Big( s\theta v_{x}^{2} + s^{3}\theta^3 v^{2}\Big) e^{2s\Phi}\,dx\,dt\\
&\leq C\Bigg(\int\!\!\!\!\!\int_{Q} \zeta^2 g^2 e^{2s\Phi}\,dx\,dt + \int\!\!\!\!\!\int_{Q_{\omega'}} \big(g^2 +  s^{3}\theta^3 z^2\big)e^{2s\Phi}\,dx\,dt\Bigg).
\end{aligned}
\end{equation}
Thus,  since $v=\zeta z$ has its support in $[0,T]\times [\alpha'',1]$, that is far away from the singularity point $x=0$, one can prove that there exists a constant $C>0$ such that:
\begin{equation}\label{star11}
\begin{aligned}
\mathfrak{J}_{\varphi,\eta,\gamma}(v)&\leq C\int\!\!\!\!\!\int_{Q} \Big( s\theta v_{x}^{2} + s^{3}\theta^3 v^{2}\Big) e^{2s\varphi}\,dx\,dt\\
& (\text{by \eqref{star9}})\\
&\leq C\left(\int\!\!\!\!\!\int_{Q} \zeta^2 g^2 e^{2s\Phi}\,dx\,dt  + \int\!\!\!\!\!\int_{Q_{\omega'}} \big(g^2 +  s^{3}\theta^3 z^2\big)e^{2s\Phi}\,dx\,dt\right).
\end{aligned}
\end{equation}

Note that
\begin{equation*}
z^2=(w+v)^2\leq 2( w^2 + v^2)\qquad \text{and} \qquad z_x^2=(w_x+v_x)^2\leq 2(w_x^2 + v_x^2).
\end{equation*}
Therefore, adding \eqref{intercarl01} and \eqref{star11}, \eqref{maincarl01} follows immediately.
\end{proof}

For our purposes in the next section, we concentrate now on a Carleman
inequality for solutions of \eqref{adjproblem} obtained via weight functions not exploding at $t=0$.
To this end, we will apply a classical argument that can be found, for instance, in \cite{FI1996} and recently in \cite{Allal2020} for a degenerate parabolic equation with memory.
More precisely, let us consider the function:
\begin{equation}\label{nu}
\nu(t)=  \left\{
\begin{array}{ll}
\theta(\frac{T}{2}), & t \in \left[0,\ds\frac{T}{2}\right], \\
\theta(t), & t \in \left[\ds\frac{T}{2},T\right],
\end{array}
\right.
\end{equation}
and the following associated weight functions:
\begin{equation}\label{nwieght}
\begin{aligned}
&\tilde{\varphi}(t,x):=\nu(t) \psi(x),\qquad \tilde{\Phi}(t,x):=\nu(t) \Psi(x), \\
&\hat{\Phi}(t):=\displaystyle\max_{x\in[0,1]}\tilde{\Phi}(t,x), \quad \hat{\varphi}(t):=\displaystyle\max_{x\in[0,1]}\tilde{\varphi}(t,x)\quad \text{and} \quad \check{\varphi}(t):=\displaystyle\min_{x\in[0,1]}\tilde{\varphi}(t,x).
\end{aligned}
\end{equation}

Now we are ready to state and prove this new modified Carleman estimate for the adjoint problem \eqref{adjproblem}.
\begin{lemma}\label{modifiedcarl}
Assume that $\mu\leq \ds\frac{1}{4}$. Then, there exist two positive constants $C$ and $s_0$ such that every solution $z$ of \eqref{adjproblem} satisfies, for all $s \geq s_0$
\begin{align}\label{modcarl}
&\|e^{s\hat{\varphi}(0)}z(0)\|_{L^2(0,1)}^2 +  \int\!\!\!\!\!\int_{Q} \nu z^{2} e^{2s\tilde{\varphi}}\,dx\,dt  \notag\\
&\leq C e^{2s[\hat{\varphi}(0)-\check{\varphi}(\frac{5T}{8})]} \Big(\int\!\!\!\!\!\int_{Q} g^2 e^{2s\tilde{\Phi}}\,dx\,dt +
\int\!\!\!\!\!\int_{Q_{\omega}}  s^{3}\nu^3 z^{2} e^{2s\tilde{\Phi}}\,dx\,dt\Big).
\end{align}
\end{lemma}
\begin{proof}
By the definitions of $\nu$ and $\tilde{\varphi}$ and using Proposition \ref{maincarleman}, it results that there exists a positive constant
$C$ such that all the solutions to equation \eqref{adjproblem}
satisfy
\begin{align}\label{est0}
\int_{\frac{T}{2}}^{T}\!\!\!\!\int_{0}^{1} \nu z^2 e^{2s\tilde{\varphi}}\,dx\,dt &= \int_{\frac{T}{2}}^{T}\!\!\!\!\int_{0}^{1} \theta z^2 e^{2s\varphi}\,dx\,dt \notag
\leq C\int_{\frac{T}{2}}^{T}\!\!\!\!\int_{0}^{1}s\theta \frac{z^2}{x^{\gamma}} e^{2s\varphi}\,dx\,dt \notag\\
&\leq C \Big(\int\!\!\!\!\!\int_{Q} g^2 e^{2s\Phi}\,dx\,dt + \int\!\!\!\!\!\int_{Q_{\omega'}}  s^{3}\theta^3 z^{2} e^{2s\Phi}\,dx\,dt\Big).
\end{align}

Let us introduce a function $\tau\in\mathrm{C}^1([0,T])$ such that
$\tau=1$ in $\left[0,\ds\frac{T}{2}\right]$ and $\tau \equiv 0$ in $\left[\ds\frac{5T}{8},T\right]$. Denote $\tilde{\tau}= e^{s\hat{\varphi}(0)}\sqrt{\nu} \tau$, where
$e^{s\hat{\varphi}(0)}= \displaystyle\max_{0\leq t\leq T} e^{s\hat{\varphi}(t)}$.

Let $\tilde{z}=\tilde{\tau}z$, then $\tilde{z}$ satisfies
\begin{equation}\label{zproblem}
\left\{
  \begin{array}{ll}
-\tilde{z}_t - \tilde{z}_{xx} - \ds\frac{\mu}{x^2}\tilde{z}  =-\tilde{\tau}_t z + \tilde{\tau}g, & (t,x) \in Q,\\
\tilde{z}(t,0)=\tilde{z}(t,1)=0, & t \in (0,T), \\
\tilde{z}(T,x)=0, & x \in (0,1).
  \end{array}
\right.
\end{equation}
Thanks to the estimate of $\sup_{t \in [0,T]}
\|\tilde z(t)\|^2_{L^2(0,1)}$ (see the energy estimate \eqref{stima2w}), we have
\begin{equation*}
\|\tilde{z}(0)\|_{L^2(0,1)}^2 +  \|\tilde{z}\|_{L^2(Q)}^2
\leq C \int\!\!\!\!\!\int_{Q}(\tilde{\tau}_t z + \tilde{\tau}g)^2\,dx\,dt,
\end{equation*}
which implies
\begin{equation*}
\nu(0)\|e^{s\hat{\varphi}(0)}z(0)\|_{L^2(0,1)}^2 +  \|e^{s\hat{\varphi}(0)}\sqrt{\nu}\tau z\|_{L^2(Q)}^2
\leq C \int\!\!\!\!\!\int_{Q}(\tilde{\tau}_t z + \tilde{\tau}g)^2\,dx\,dt.
\end{equation*}
By using the boundedness of $\theta$ in $\left[\ds\frac{T}{2},\ds\frac{5T}{8}\right]$, the definitions of $\tau$ and  of $\nu$ in $\left[\ds0,\ds\frac{5T}{8}\right]$ and the fact that $\nu_t(t)=0$ in $\left[0,\ds\frac{T}{2}\right]$ and $\tau(t)=0$ in $\left[\ds\frac{5T}{8},T\right]$, it holds that

\begin{align*}
&\bar{c} \left(\|e^{s\hat{\varphi}(0)}z(0)\|_{L^2(0,1)}^2 +  \int_{0}^{\frac{5T}{8}}\!\!\!\!\int_{0}^{1} \nu \tau^2 z^{2} e^{2s\hat{\varphi}}\,dx\,dt\right)\\
&\leq \nu(0)\|e^{s\hat{\varphi}(0)}z(0)\|_{L^2(0,1)}^2 +  \int_{0}^{\frac{5T}{8}}\!\!\!\!\int_{0}^{1} \nu \tau^2 z^{2} e^{2s\hat{\varphi}}\,dx\,dt\\
&\leq C\Big(\int_{\frac{T}{2}}^{\frac{5T}{8}}\!\!\!\!\int_{0}^{1}(\theta^2(t) + \theta(t))z^2 e^{2s\hat{\varphi}(0)}\,dx\,dt +\int_{0}^{\frac{5T}{8}}\!\!\!\!\int_{0}^{1} \nu g^2 e^{2s\hat{\varphi}(0)}\,dx\,dt\Big)\\
&\leq C\Big(\int_{\frac{T}{2}}^{\frac{5T}{8}}\!\!\!\!\int_{0}^{1}z^2 e^{2s\hat{\varphi}(0)}\,dx\,dt +\int_{0}^{\frac{5T}{8}}\!\!\!\!\int_{0}^{1}g^2 e^{2s\hat{\varphi}(0)}\,dx\,dt\Big),
\end{align*}
where $\bar{c}:= \min \{\nu(0), 1\}$.
That is,
\begin{align*}
&\|e^{s\hat{\varphi}(0)}z(0)\|_{L^2(0,1)}^2 +  \int_{0}^{\frac{T}{2}}\!\!\!\!\int_{0}^{1} \nu  z^{2} e^{2s\tilde{\varphi}}\,dx\,dt \notag\\
&\leq C \Big(\int_{\frac{T}{2}}^{\frac{5T}{8}}\!\!\!\!\int_{0}^{1}z^2 e^{2s(\hat{\varphi}(0)-\tilde{\varphi})}e^{2s\tilde{\varphi}}\,dx\,dt +\int_{0}^{\frac{5T}{8}}\!\!\!\!\int_{0}^{1}g^2 e^{2s(\hat{\varphi}(0)-\tilde{\varphi})}e^{2s\tilde{\varphi}}\,dx\,dt \Big).
\end{align*}
Observe that
\begin{equation*}
\check{\varphi}\left(\frac{5T}{8}\right) \leq \tilde{\varphi} \;\quad \text{in}\quad \left(0,\frac{5T}{8}\right)\times(0,1)
\end{equation*}
so that,
\begin{align}\label{est1}
&\|e^{s\hat{\varphi}(0)}z(0)\|_{L^2(0,1)}^2 +  \int_{0}^{\frac{T}{2}}\!\!\!\!\int_{0}^{1} \nu  z^{2} e^{2s\tilde{\varphi}}\,dx\,dt \notag\\
&\leq C e^{2s(\hat{\varphi}(0)-\check{\varphi}(\frac{5T}{8}))} \Big(\int_{\frac{T}{2}}^{\frac{5T}{8}}\!\!\!\!\int_{0}^{1}z^2 e^{2s\tilde{\varphi}}\,dx\,dt +\int_{0}^{\frac{5T}{8}}\!\!\!\!\int_{0}^{1}g^2 e^{2s\tilde{\varphi}}\,dx\,dt \Big).
\end{align}
As in \eqref{est0}, one can prove that there exists a positive constant
$C$ such that
\begin{equation*}
\int_{\frac{T}{2}}^{\frac{5T}{8}}\!\!\!\!\int_{0}^{1}z^2 e^{2s\tilde{\varphi}}\,dx\,dt
\leq C \Big(\int\!\!\!\!\!\int_{Q} g^2 e^{2s\Phi}\,dx\,dt +
\int\!\!\!\!\!\int_{Q_{\omega}}  s^{3}\theta^3 z^{2} e^{2s\Phi}\,dx\,dt\Big).
\end{equation*}
Using this last inequality in \eqref{est1}, we have
\begin{align}\label{est2}
\|e^{s\hat{\varphi}(0)}z(0)\|_{L^2(0,1)}^2 &+  \int_{0}^{\frac{T}{2}}\!\!\!\!\int_{0}^{1} \nu  z^{2} e^{2s\tilde{\varphi}}\,dx\,dt
\leq C e^{2s(\hat{\varphi}(0)-\check{\varphi}(\frac{5T}{8}))} \Big(\int\!\!\!\!\!\int_{Q} g^2 e^{2s\Phi}\,dx\,dt \notag\\
&+\int\!\!\!\!\!\int_{Q_{\omega}}  s^{3}\theta^3 z^{2} e^{2s\Phi}\,dx\,dt + \int_{0}^{\frac{5T}{8}}\!\!\!\!\int_{0}^{1}g^2 e^{2s\tilde{\varphi}}\,dx\,dt \Big).
\end{align}
From \eqref{star10} and by the definition of the modified
weights, notice that, in particular $\tilde{\varphi} \leq \tilde{\Phi}$ and $\Phi \leq \tilde{\Phi}$ in $Q$.
This, together with \eqref{est0} and \eqref{est2}, implies that
\begin{align}\label{est3}
\|e^{s\hat{\varphi}(0)}z(0)\|_{L^2(0,1)}^2 &+  \int_{0}^{T}\!\!\!\!\int_{0}^{1} \nu  z^{2} e^{2s\tilde{\varphi}}\,dx\,dt
\leq C e^{2s(\hat{\varphi}(0)-\check{\varphi}(\frac{5T}{8}))} \Big(\int\!\!\!\!\!\int_{Q} g^2 e^{2s\tilde{\Phi}}\,dx\,dt \notag\\
&+\int\!\!\!\!\!\int_{Q_{\omega}}  s^{3}\theta^3 z^{2} e^{2s\Phi}\,dx\,dt \Big).
\end{align}
To conclude, it suffices to remark that for $c>0$, the function $x \mapsto s^3 e^{-cs}$ is nonincreasing for $s$ sufficiently large.
So, since $\nu(t)\leq \theta(t)$ by taking $s$ large enough, one has
$$
s^{3}\theta^3 e^{2s\Phi} \leq s^{3}\nu^3 e^{2s\tilde{\Phi}},
$$
which, together with \eqref{est3}, provides the desired inequality.
\end{proof}
\subsection{Null controllability result}
\label{NCSHE}
Following the classical method as in \cite{FI1996}, with the modified Carleman inequality proved in the previous subsection, we can get a null controllability result for
\eqref{problem01}. However, as explained in \cite{TG2016}, this null controllability result cannot help to solve the controllability for integro-differential equations. Indeed, we will need to prove the null controllability of the singular heat equation \eqref{problem01}, for more regular solutions. For this reason, to formulate our results we introduce the following function space where the controllability will be solved:
\begin{equation*}
X_{s}:=\big\{y\in \mathcal{Z}: e^{-s\tilde{\Phi}}y \in L^2(Q) \big\}
\end{equation*}
equipped with the norm
\begin{align*}
\| y \|_{X_{s}}:=  \|e^{-s\tilde{\Phi}}y \|_{L^2(Q)}.
\end{align*}
Observe that, since $\tilde{\Phi}<0$, we have that the function $e^{-s\tilde{\Phi}}$ tends to $+\infty$ for $t \rightarrow T^-$.
Therefore, $y\in X_{s}$ requires that the solution $y$ has more regularity than the one in Lemma \ref{prop}. Moreover,

\begin{equation}\label{NC}
\text{ if } y\in X_{s} \;\text{  then }\;
y(T,x)=0 \text{ in } (0,1).
\end{equation}
From now on, we denote by $s_0$ the parameter defined in Lemma \ref{modifiedcarl}. Our first result, stated as follows, ensures the null controllability for \eqref{problem01}.

\begin{theorem}\label{firstresult}
Assume that $\mu\leq \ds \frac{1}{4}$ and $y_0\in H_0^{1,\mu}(0,1)$. If
$e^{-s\tilde{\varphi}}f \in L^2(Q)$ with $s\geq s_0$, then there exists a control function $u \in L^2(Q)$, such that the associated solution $y$
of \eqref{problem01} belongs to $X_{s}$.

Moreover, there exists a positive constant $C$ such that $y$ satisfies the following
estimate:
\begin{equation}\label{mainestimate}
\begin{aligned}
&\int\!\!\!\!\!\int_{Q} y^2 e^{-2s\tilde{\Phi}}\,dx\,dt +
\int\!\!\!\!\!\int_{Q_{\omega}}  s^{-3} \nu^{-3} u^{2} e^{-2s\tilde{\Phi}}\,dx\,dt \\
&\leq C e^{2s[\hat{\varphi}(0)-\check{\varphi}(\frac{5T}{8})]} \Big(\int\!\!\!\!\!\int_{Q} f^{2} e^{-2s\tilde{\varphi}}\,dx\,dt+ \|y_0e^{-s\hat{\varphi}(0)}\|_{L^2(0,1)}^2\Big).
\end{aligned}
\end{equation}
\end{theorem}
\begin{proof}
Following the ideas in \cite{cara,TG2016}, fixed $s \geq s_0$, let us consider the functional
\begin{equation}\label{extremal}
J(y,u)= \Big(\int\!\!\!\!\!\int_{Q} y^2 e^{-2s\tilde{\Phi}}\,dx\,dt +
\int\!\!\!\!\!\int_{Q_{\omega}}  s^{-3}\nu^{-3} u^{2} e^{-2s\tilde{\Phi}}\,dx\,dt \Big),
\end{equation}
where $(y,u)$ satisfies
\begin{equation}\label{problem001}
\left\{
  \begin{array}{ll}
y_t - y_{xx} - \ds\frac{\mu}{x^2}y  =f + 1_{\omega} u(t), & (t,x) \in Q,\\
y(t,0)=y(t,1)=0, & t \in (0,T), \\
y(0,x)=y_0(x),\quad y(T,x)=0 & x \in (0,1),
  \end{array}
\right.
\end{equation}
with $u\in L^2(Q)$.

By means of standard arguments, it is easy to prove (see \cite{Lions,Lionsb})  that $J$ attains its minimizer at a unique point  denoted as $(\bar{y},\bar{u})$.

We set
$$
L_{\mu}y:= y_t - y_{xx} - \frac{\mu}{x^2}y \qquad \text{in}\quad Q.
$$
We will first prove that there exists a dual variable $\bar{z}$ such that
\begin{equation}\label{step0}
\left\{
  \begin{array}{ll}
\bar{y}=e^{2s\tilde{\Phi}} L_{\mu}^{\star} \bar{z}, &\quad \text{in}\quad Q,\\
\bar{u}= - s^3 \nu^3 e^{2s\tilde{\Phi}} \bar{z},  &\quad \text{in}\quad (0,T)\times \omega,\\
\bar{z}=0, &\quad \text{on}\quad (0,T)\times\{0,1\},
  \end{array}
\right.
\end{equation}
where $L_{\mu}^{\star}$ is the (formally) adjoint operator of $L_{\mu}$.

Let us start by introducing the following linear space
$$
\mathcal{P}_0=\big\{z \in C^{\infty}(\overline{Q}): z=0 \quad \text{on}\quad (0,T)\times\{0,1\}\big\},
$$
and introduce the bilinear form $a$:
\begin{equation*}
a(z_1,z_2)= \int\!\!\!\!\!\int_{Q} e^{2s\tilde{\Phi}} L_{\mu}^{\star} z_1 L_{\mu}^{\star} z_2\,dx\,dt
+\int\!\!\!\!\!\int_{Q_{\omega}}  s^{3}\nu^3 e^{2s\tilde{\Phi}} z_1 z_2\,dx\,dt,\quad \forall \; z_1,z_2\in \mathcal{P}_0.
\end{equation*}
Then, if the functions $\bar{y}$ and $\bar{u}$ given by \eqref{step0} satisfy the parabolic problem \eqref{problem001}, we
must have
\begin{equation}\label{step00}
a(\bar{z},z)= \int\!\!\!\!\!\int_{Q} f z\,dx\,dt + \int_{0}^{1}y_0 z(0)\,dx,\quad \forall\; z\in \mathcal{P}_0.
\end{equation}
The key idea in this proof is to show that there exists exactly one $\bar{z}$
satisfying \eqref{step00} in an appropriate class. We will then define $\bar{y}$ and $\bar{u}$ using \eqref{step0}
and we will check that the couple $(\bar{y},\bar{u})$ fulfills the desired properties.

Observe that the modified Carleman inequality \eqref{modcarl} holds for all $z \in \mathcal{P}_{0}$. Consequently,
\begin{equation}\label{step000}
\|e^{s\hat{\varphi}(0)}z(0)\|_{L^2(0,1)}^2 +  \int\!\!\!\!\!\int_{Q} \nu z^{2} e^{2s\tilde{\varphi}}\,dx\,dt
\leq C e^{2s[\hat{\varphi}(0)-\check{\varphi}(\frac{5T}{8})]} a(z,z).
\end{equation}
In particular, $a(\cdot,\cdot)$ is a strictly positive and symmetric bilinear form, that is, $a(\cdot,\cdot)$ is a scalar product in $\mathcal{P}_{0}$.

Denote by $\mathcal{P}$ the Hilbert space which is the completion of $\mathcal{P}_{0}$ with respect to the norm associated
to $a(\cdot,\cdot)$ (which we denote by $\|\cdot\|_{\mathcal{P}}$). Let us now consider the linear form $l$, given by
\begin{equation*}
l(z)=\int\!\!\!\!\!\int_{Q} f z\,dx\,dt + \int_{0}^{1}y_0 z(0)\,dx, \quad
\forall \; z\in \mathcal{P}.
\end{equation*}
By the Cauchy-Schwarz inequality and in view of \eqref{step000}, we have that
\begin{align*}
|l(z)| &\leq \left\|f\frac{e^{-s\tilde{\varphi}}}{\sqrt{\nu}}\right\|_{L^2(Q)}  \|z \sqrt{\nu}e^{s\tilde{\varphi}}\|_{L^2(Q)}
 + \|y_0 e^{-s\hat{\varphi}(0)}\|_{L^2(0,1)} \|z(0)e^{s\hat{\varphi}(0)} \|_{L^2(0,1)} \\
&\leq C e^{s[\hat{\varphi}(0)-\check{\varphi}(\frac{5T}{8})]} \Big(\|f e^{-s\tilde{\varphi}}\|_{L^2(Q)} +  \|y_0 e^{-s\hat{\varphi}(0)}\|_{L^2(0,1)}\Big) \|z\|_{\mathcal{P}},
\end{align*}
and then $l$ is a linear continuous form on $\mathcal{P}$. Hence, in view of Lax-Milgram's Lemma, there
exists one and only one $\bar{z}\in \mathcal{P}$ satisfying
\begin{equation}\label{step001}
a(\bar{z},z)= l(z),\quad \forall \;z\in \mathcal{P}.
\end{equation}
Moreover, we have
\begin{equation}\label{step002}
\|\bar{z}\|_{\mathcal{P}}\leq C e^{s[\hat{\varphi}(0)-\check{\varphi}(\frac{5T}{8})]} \Big(\|f e^{-s\tilde{\varphi}}\|_{L^2(Q)} +  \|y_0 e^{-s\hat{\varphi}(0)}\|_{L^2(0,1)}\Big).
\end{equation}
Let us set
\begin{equation}\label{step003}
\bar{y}=e^{2s\tilde{\Phi}} L_{\mu}^{\star} \bar{z}\quad \text{and}\quad
\bar{u}= - 1_{\omega} s^3 \nu^3 e^{2s\tilde{\Phi}} \bar{z}.
\end{equation}
With these definitions and by \eqref{step002}, it is easy to check that $\bar{y}$ and $\bar{u}$ satisfy
\begin{equation}\label{step004}
\begin{aligned}
&\int\!\!\!\!\!\int_{Q} \bar{y}^2 e^{-2s\tilde{\Phi}}\,dx\,dt +
\int\!\!\!\!\!\int_{Q_{\omega}} s^{-3} \nu^{-3} \bar{u}^{2} e^{-2s\tilde{\Phi}}\,dx\,dt \\
&\leq C e^{2s[\hat{\varphi}(0)-\check{\varphi}(\frac{5T}{8})]} \Big(\|f e^{-s\tilde{\varphi}}\|_{L^2(Q)}^2 +  \|y_0 e^{-s\hat{\varphi}(0)}\|_{L^2(0,1)}^2\Big),
\end{aligned}
\end{equation}
which implies \eqref{mainestimate}.

It remains to check that $\bar{y}$ is the solution of \eqref{problem001} corresponding to $\bar{u}$.  First of all, it is immediate that $\bar{y}\in X_{s}$ and
$\bar{u}\in L^2(Q)$. Denote by $\tilde{y}$ the (weak) solution of \eqref{problem01}
associated to the control function $u=\bar{u}$, then $\tilde{y}$ is also the unique solution of \eqref{problem01} defined by transposition. In other words, $\tilde{y}$ is the unique function in $L^2(Q)$ satisfying
\begin{equation}\label{trans}
\int\!\!\!\!\!\int_{Q} \tilde{y} h\,dx\,dt= \int\!\!\!\!\!\int_{Q} 1_{\omega} \bar{u} z\,dx\,dt + \int\!\!\!\!\!\int_{Q} f z\,dx\,dt + \int_{0}^{1}y_0 z(0)\,dx,\quad \forall \;h \in L^2(Q),
\end{equation}
where $z$ is the solution to
\begin{equation*}
\left\{
  \begin{array}{ll}
-z_t - z_{xx} - \frac{\mu}{x^2}z  = h, & (t,x) \in Q,\\
z(t,0)=z(t,1)=0, & t \in (0,T), \\
z(T,x)=0 & x \in (0,1).
  \end{array}
\right.
\end{equation*}
According to \eqref{step001} and \eqref{step003}, we see that $\bar{y}$ also satisfies \eqref{trans}. Therefore, $\bar{y}=\tilde{y}$. Consequently, the control $\bar{u}\in L^2(\omega\times(0,T))$ drives the state $\bar{y}\in X_{s}$ exactly to zero at time $T$.

\end{proof}

\section{Singular heat equation with memory}
\label{SHEM}
Prior to null controllability is the well-posedness of problem \eqref{problem}.
From the results in \cite{lorenzi}, we recall that in the nonsingular case $(\mu=0)$, it is well known that the heat operator with memory
gives rise to well-posed Cauchy-Dirichlet problems. Likewise in \cite{lorenzi}, by an application of the Contraction Mapping Principle and invoking Theorem \ref{prop}, we have that \eqref{problem} is well-posed in the following sense:
\begin{proposition}\label{well-posed}
Assume that $\mu\leq \ds\frac{1}{4}$. If $y_0 \in L^2(0,1)$ and $u\in L^2(Q)$, then there exists a unique solution $y$ of
\eqref{problem} such that
\[y \in C\big([0,T]; L^2(0,1)\big) \cap L^2 \big(0,T;H_{0}^{1,\mu}(0,1)\big).\]
\end{proposition}
Now, we pass to derive our main result, which concerns the null controllability of the singular heat equation with memory \eqref{problem}.
Hence, in what follows, we assume that the function $a$ satisfies
\begin{equation}\label{conditiona}
e^{\frac{4^ks\mathfrak{c}d}{T^k(T-t)^k}}a \in L^{\infty}((0,T)\times Q),
\end{equation}
where $\mathfrak{c},d,k$ are the constants defined in \eqref{weightfunc} and $s$ is the same of Theorem \ref{firstresult}.

\begin{remark}
It is worth mentioning that, from the results in Guerrero and Imanuvilov \cite{guerrero}, it seems that the null controllability property
of parabolic equations with memory may fail without any additional conditions on the kernel. On the other hand, observe that the condition \eqref{conditiona} just restricts the function $a$ very near $T$, which is due to the fact that the function $\nu$ blows up only at $t=T$.
\end{remark}

For our proof, we are going to employ a fixed point strategy.
For $R>0$, we define
$$
X_{s,R}=\big\{w \in X_{s}:\, \|e^{-s\tilde{\Phi}}w\|_{L^2(Q)}\leq R \big\},
$$
which is a bounded, closed, and convex subset of $L^2(Q)$.

For any $w\in X_{s,R}$, let us consider the control problem
\begin{equation}\label{problem20}
\left\{
  \begin{array}{ll}
y_t - y_{xx} - \ds\frac{\mu}{x^2}y  =\int_{0}^{t}a(t,s,x)w(s,x)\,ds + 1_{\omega} u, & (t,x) \in Q,\\
y(t,0)=y(t,1)=0, & t \in (0,T), \\
y(0,x)=y_0(x), & x \in (0,1).
  \end{array}
\right.
\end{equation}
By Theorem \ref{firstresult} we first derive a null controllability result for \eqref{problem20}; then, as a second step, we will obtain the same controllability result for \eqref{problem} applying Kakutani's fixed point Theorem.

Our main result is thus the following.
\begin{theorem}\label{mainresult}
Assume that $\mu\leq \ds\frac{1}{4}$. If the function $a$ satisfies \eqref{conditiona}, then for any $y_0\in H_{0}^{1,\mu}(0,1)$, there exists a control function $u\in L^2(Q)$ such that the associated solution $y\in \mathcal{Z}$ of \eqref{problem} satisfies
\begin{equation}\label{ncresult}
y(T,\cdot)=0 \qquad \text{in}\quad (0,1).
\end{equation}
\end{theorem}
\begin{proof} Setting $C_0:=\frac{4^k\mathfrak{c}d}{T^k}$, by \eqref{conditiona} and the estimate $e^{-s\tilde{\varphi}}\leq e^{\frac{sC_0}{(T-t)^k}}$, we get that
\begin{align}\label{r1}
&\int\!\!\!\!\!\int_{Q}  \Big( e^{-s\tilde{\varphi}} \int_{0}^{t} a(t,s,x)  w(s,x)\,ds\Big)^2 \,dx\,dt
\leq C \int\!\!\!\!\!\int_{Q}  \int_{0}^{t} e^{\frac{2C_0 s}{(T-t)^k}} a^2(t,s,x) w^2(s,x)\,ds\,dx\,dt \notag\\
&\leq C \int\!\!\!\!\!\int_{Q} w^2\,dx\,dt \notag \leq C \Big(\displaystyle\sup_{(t,x) \in \overline{Q}} e^{2s\tilde{\Phi}} \Big) \int\!\!\!\!\!\int_{Q} e^{-2s\tilde{\Phi}} w^2\,dx\,dt \notag\leq C R^2<+\infty.
\end{align}
(recall that $w\in X_{s,R}$).
Thus, the result in Theorem \ref{firstresult} holds for
the equation \eqref{problem20}, i.e. for any $y_0\in H_{0}^{1,\mu}(0,1)$, there exists a control function $u\in L^2(Q)$ such that the associated solution $y$ of \eqref{problem20} is in $X_{s}$ and
\begin{equation*}
y(T,\cdot)=0 \qquad \text{in}\quad (0,1).
\end{equation*}
Let us now introduce, for every $w\in X_{s,R}$, the multivalued map
$$\Lambda: X_{s,R} \subset X_{s} \rightarrow 2^{X_{s}}$$
with
\begin{align*}
\Lambda(w)=\displaystyle\Big\{&y\in X_s: \text{ for some}\; u \in L^2(Q)\; \text{satisfying} \\
&\int\!\!\!\!\!\int_{Q_{\omega}}  s^{-3} \nu^{-3} u^{2} e^{-2s\tilde{\Phi}}\,dx\,dt
\leq C e^{2s[\hat{\varphi}(0)-\check{\varphi}(\frac{5T}{8})]} \Big(R^2+\int_{0}^{1}y_0^2 e^{-2s\hat{\varphi}(0)}\,dx\,dt\Big)\\
&y \text{ solves }\eqref{problem20}\Big\}.
\end{align*}
Observe that if $y\in \Lambda (w)$, then $y(T, \cdot)=0$ in $(0,1)$ via \eqref{NC}.

To achieve our goal, it will suffice to show that $\Lambda$ possesses at least one fixed point.
To this purpose, we shall apply Kakutani's fixed point Theorem (see \cite[Theorem 2.3]{cara}).

It is readily seen that $\Lambda(w)$ is a nonempty, closed and convex subset of $L^2(Q)$
for every $w\in X_{s,R}$. Then, we prove that $\Lambda(X_{s,R})\subset X_{s,R}$ with sufficiently large $R>0$.
By \eqref{mainestimate} and condition \eqref{conditiona}, and arguing as before we have
\begin{align*}
&\int\!\!\!\!\!\int_{Q} y^2 e^{-2s\tilde{\Phi}}\,dx\,dt +
\int\!\!\!\!\!\int_{Q_{\omega}}  s^{-3} \nu^{-3} u^{2} e^{-2s\tilde{\Phi}}\,dx\,dt \\
&\leq C e^{2s[\hat{\varphi}(0)-\check{\varphi}(\frac{5T}{8})]} \Big(\int\!\!\!\!\!\int_{Q}  e^{-2s\tilde{\varphi}}  \Big(\int_{0}^{t} a(t,s,x)  w(s,x)\,ds\Big)^2 \,dx\,dt + e^{-2s\hat{\varphi}(0)} \int_0^1 y_0^2\,dx\Big) \\
&\leq   C e^{2s[\hat{\varphi}(0)-\check{\varphi}(\frac{5T}{8})]} \Big( \int\!\!\!\!\!\int_{Q} w^2(t,x)\,dx\,dt + e^{-2s\hat{\varphi}(0)}  \int_0^1 y_0^2\,dx\Big)\\
&\leq  C e^{2s[\hat{\varphi}(0)-\check{\varphi}(\frac{5T}{8})]} \Big(\displaystyle\sup_{(t,x) \in \overline{Q}} e^{2s\tilde{\Phi}} \Big) \Big( \int\!\!\!\!\!\int_{Q} e^{-2s\tilde{\Phi}(t,x)} w^2(t,x)\,dx\,dt \Big) + C e^{-2s\check{\varphi}(\frac{5T}{8})} \int_0^1 y_0^2\,dx.
\end{align*}
By virtue of $\hat \varphi(0) \le \hat \Phi(0)$ and $ \tilde{\Phi} \leq \hat{\Phi}(0)$ in $Q$, we get
\begin{align}\label{starstima}
&\int\!\!\!\!\!\int_{Q} y^2 e^{-2s\tilde{\Phi}}\,dx\,dt +
\int\!\!\!\!\!\int_{Q_{\omega}}  s^{-3} \nu^{-3} u^{2} e^{-2s\tilde{\Phi}}\,dx\,dt \notag \\
&\leq  C e^{s[2\hat{\varphi}(0)-2\check{\varphi}(\frac{5T}{8})+2\hat{\Phi}(0)]} \int\!\!\!\!\!\int_{Q} e^{-2s\tilde{\Phi}(t,x)} w^2(t,x)\,dx\,dt
+ C e^{-2s\check{\varphi}(\frac{5T}{8})} \int_0^1 y_0^2\,dx \notag \\
&\leq  C e^{s[4\hat{\Phi}(0)-2\check{\varphi}(\frac{5T}{8})]} R^2 + C e^{-2s\check{\varphi}(\frac{5T}{8})} \int_0^1  y_0^2\,dx.
\end{align}
Now, choosing the constant $\mathfrak{c}$ (see \eqref{weightfunc}) in the interval
$$
\left(\frac{e^{2 \rho \|\sigma\|_{\infty}} -1}{d-1}, \frac{16}{15}\frac{e^{2 \rho \|\sigma\|_{\infty}} -e^{\rho \|\sigma\|_{\infty}}}{d-1}\right),
$$
which is not empty for $\rho$ sufficiently large, we have
\begin{align*}
2\hat{\Phi}(0)-\check{\varphi}\left(\frac{5T}{8}\right)
&=\left(\frac{4}{T^2}\right)^k\Big[2(e^{\rho\|\sigma\|_{\infty}}-e^{2\rho\|\sigma\|_{\infty}}) +\mathfrak{c} d\left (\frac{16}{15}\right)^k \Big]\\
&< \left(\frac{4}{T^2}\right)^k \left(-2 +  \frac{d}{d-1} \left(\frac{16}{15}\right)^{k+1} \right) (e^{2\rho\|\sigma\|_{\infty}}-e^{\rho\|\sigma\|_{\infty}}).
\end{align*}
Therefore, taking the parameters  $d$ and $k$ defined in \eqref{weightfunc} in such a way that  $d>3$ and $2<k<\frac{\ln(4/3)}{\ln(16/15)}-1$, we infer that
$$
2\hat{\Phi}(0)-\check{\varphi}(\frac{5T}{8})<0.
$$
Hence for $s$ sufficiently large, increasing the parameter $s_0$ if necessary, we obtain
\begin{equation*}
\int\!\!\!\!\!\int_{Q} y^2 e^{-2s\tilde{\Phi}}\,dx\,dt +
\int\!\!\!\!\!\int_{Q_{\omega}}  s^{-3} \nu^{-3} u^{2} e^{-2s\tilde{\Phi}}\,dx\,dt
\leq \frac{1}{2} R^2 + C e^{-2s\check{\varphi}(\frac{5T}{8})} \int_0^1  y_0^2\,dx.
\end{equation*}
Then, for  $s$ and $R$ large enough, we obtain
\begin{equation*}
\int\!\!\!\!\!\int_{Q} y^2 e^{-2s\tilde{\Phi}}\,dx\,dt
\leq R^2.
\end{equation*}

It follows that $\Lambda(X_{s,R}) \subset X_{s,R}$.
Furthermore, let $\{w_n\}$ be a sequence of $X_{s,R}$. The regularity assumption on $y_0$
and Theorem \ref{prop}, imply that the associated solutions $\{y_n\}$ are bounded in  $H^1\big(0,T; L^2(0,1)\big) \cap L^2\big(0,T;D(A)\big)$.
Therefore, $\Lambda(X_{s,R})$ is a relatively compact subset of $L^2(Q)$ by
the Aubin-Lions Theorem \cite{simon}.

In order to conclude, we have to prove that $\Lambda$ is upper-semicontinuous under the $L^2$ topology.
First, observe that for any $w\in X_{s,R}$, we have at least $u \in L^2(Q)$ such that the corresponding solution $y\in X_{s,R}$.
Hence, taking $\{w_n\}$ a sequence in $X_{s,R}$, we can find a sequence of controls $\{u_n\}$ such that the corresponding solutions
$\{y_n\}$ is in $L^2(Q)$. Thus, let $\{w_n\}$ be a sequence satisfying $w_n \rightarrow w$ in $X_{s,R}$ and
$y_n \in \Lambda(w_n)$ such that $y_n \rightarrow y$ in $L^2(Q)$. We must prove that
$y \in \Lambda(w)$.
For every $n$, we have a control $u_n \in L^2(Q)$ such that the system
\begin{equation}\label{sys_n}
\left\{
\begin{array}{lll}
\displaystyle y_{n,t} - y_{n,xx} -\frac{\mu}{x^2}y_n = \int_0^t a(t,s,x)  w_n(s,x)\,ds +1_{\omega} u_n, &  & (t, x) \in Q,\\
y_n(t, 0)= y_n(t,1)= 0, & & t \in (0,T),  \\
y_n(0,x)= y_{0}(x),  & & x \in  (0,1)
\end{array}
\right.
\end{equation}
has a least one solution $y_n\in L^2(Q)$ that satisfies
\begin{equation*}
y_n(T,\cdot)=0 \qquad \text{in}\quad (0,1).
\end{equation*}
From Theorem \ref{prop} and \eqref{starstima}, it follows (at least for a subsequence) that
\begin{align*}
 u_n \rightarrow u \quad &\text{weakly in} \; L^2(Q),\notag \\
 y_n \rightarrow y \quad &\text{weakly in} \; H^1\big(0,T; L^2(0,1)\big) \cap L^2\big(0,T;D(A)\big),\\
 &\text{strongly in} \; C(0, T; L^2(0,1)).
\end{align*}
Passing to the limit in \eqref{sys_n}, we obtain a control $u\in L^2(Q)$ such that the corresponding solution $y$ to \eqref{problem20}
satisfies \eqref{ncresult}. This shows that $y\in \Lambda(w)$ and, therefore, the map $\Lambda$ is upper-semicontinuous.

Hence, the multivalued map $\Lambda$ possesses at least one fixed point, i.e., there exists $y \in X_{s,R}$ such that $y\in \Lambda(y)$.
By the definition of $\Lambda$, this implies that there exists at least one pair $(y,u)$ satisfying the conditions of Theorem \ref{mainresult}.
The uniqueness of $y$ follows by Proposition \ref{well-posed}.
This ends the proof of Theorem \ref{mainresult}.

\end{proof}
As a consequence of the previous theorem one has the next result.
\begin{theorem}\label{Thm_null_Control_memo_2}
Assume that $\mu\leq \ds\frac{1}{4}$. If the function $a$ satisfies \eqref{conditiona}, then for any $y_0 \in L^2(0,1)$, there exists a control function $u\in L^2(Q)$ such that the associated solution $y\in \mathcal{W}$ of \eqref{problem} satisfies
\begin{equation*}
y(T,\cdot)=0 \qquad \text{in}\quad (0,1).
\end{equation*}
\end{theorem}
\begin{proof}
Consider the following singular parabolic problem:
\begin{equation*}
\left\{
\begin{array}{lll}
w_t - w_{xx} - \ds\frac{\mu}{x^2}w = \int_{0}^{t} a(t,s,x)  w(s,x) \, ds  &  & (t, x) \in \left(0, \ds\frac{T}{2}\right)\times (0, 1), \\
w(t, 0) = w(t,1)= 0, & & t \in\left(0, \ds\frac{T}{2}\right),  \\
w(0,x)= y_{0}(x),  & & x \in  (0, 1),
\end{array}
\right.
\end{equation*}
where $y_0\in L^2(0,1)$ is the initial condition in \eqref{problem}.

By Theorem \ref{prop}, the solution of this system belongs to
$$ \mathcal{W}\left(0,\frac{T}{2}\right) := L^2\left(0, \frac{T}{2}; H_{0}^{1,\mu}(0,1)\right)\cap C\left(\left[0, \frac{T}{2}\right]; L^2(0,1)\right).$$
Then, there exists $t_0 \in (0, \frac{T}{2})$ such that
$w(t_0, \cdot) := \tilde{w}(\cdot) \in H_{0}^{1,\mu}(0,1)$.

Now, we consider the following controlled parabolic system:
\begin{equation*}
\left\{
\begin{array}{lll}
z_t - z_{xx} - \ds\frac{\mu}{x^2} z= \int_{0}^{t} a(t,s,x) z(s,x)\,ds + 1_{\omega} h &  & (t, x) \in (t_0, T) \times (0, 1), \\
z(t, 0) =z(t,1)= 0, & & t \in (t_0, T),  \\
z(t_0,x)= \tilde{w}(x),  & & x \in  (0, 1).
\end{array}
\right.
\end{equation*}
We start by observing that, since Theorem \ref{mainresult} holds also in a general domain $(t_0,T)\times(0,1)$ with suitable changes,
we can see that there exists a control function  $h \in L^2((t_0, T) \times (0, 1))$ such that the associated solution
$$z \in \mathcal{Z}(t_0,T) :=  L^2(t_0, T; D(A))\cap H^1(t_0, T; L^2(0,1))\cap C\left(\left[t_0, T\right]; H_{0}^{1,\mu}(0,1)\right)$$
satisfies
$$z(T, \cdot) = 0 \qquad \text{in} \; (0, 1).$$
Finally, setting
\begin{align*}
y:=
\left\{
\begin{array}{lll}
w, & \text{in} \quad \big[0, t_0\big],\\
z, & \text{in} \quad \big[t_0, T\big]
\end{array}
\right.
 \quad  \text{and} \quad u:=
\left\{
\begin{array}{lll}
0, & \text{in} \quad \big[0, t_0\big],\\
h, & \text{in} \quad \big[t_0, T\big],
\end{array}
\right.
\end{align*}
one can prove that $y\in \mathcal{W}$ is the solution to the system \eqref{problem} corresponding to $u$ and satisfies
$$y(T, \cdot) = 0 \qquad \text{in} \; (0, 1).$$
Hence, our assertion is proved.
\end{proof}



\begin{thebibliography}{99}
\bibitem{Allal2020} B. Allal, G. Fragnelli, {\it Controllability of degenerate parabolic equation with memory}, submitted.


\bibitem{allal} B. Allal, A. Hajjaj, L. Maniar, J. Salhi,
{\it Lipschitz stability for some coupled degenerate parabolic systems with locally distributed observations of one component}, Math. Control \& Rela. Fields, doi: 10.3934/mcrf.2020014.

\bibitem{amendola} G. Amendola, M. Fabrizio, J.M. Golden,
{\it Thermodynamics of Materials With Memory: Theory and Applications}, Springer, New York, 2012.



\bibitem{barbu} V. Barbu, M. Iannelli,
{\it Controllability of the heat equation with memory}, Differential Integral Equations, \textbf{13} (2000), 1393--1412.


\bibitem{biccari2019}
U. Biccari,
{\it Boundary controllability for a one-dimensional heat equation with a singular inverse-square potential},
Math. Control Relat. Fields, \textbf{9} (2019), 191--219.

\bibitem{BZ2016}
U. Biccari, E. Zuazua, {\it Null controllability for a heat equation with a singular
inverse-square potential involving the distance to the boundary function}, J. Differential
Equations, \textbf{261} (2016), 2809--2853.


\bibitem{bloom} F. Bloom,
{\it Ill-Posed Problems for Integro-Differential Equations in Mechanics and Electromagnetic Theory},
SIAM Studies in Applied Mathematics \textbf{3}, Society for Industrial and Applied Mathematics (SIAM), Philadelphia, Pa., 1981.

\bibitem{b} H. Brezis, Functional Analysis,
Sobolev Spaces and Partial Differential Equations, Springer
Science+Business Media, LLC 2011.

\bibitem{cara} E. Fern\'{a}ndez-Cara and S. Guerrero,
{\it Global Carleman inequalities for parabolic systems and applications to null controllability},
SIAM J. Control Optim., \textbf{45} (2006), 1395--1446.


\bibitem{Cazacu2014} C. Cazacu,
{\it Controllability of the heat equation with an inverse-square potential localized on the boundary},
SIAM J. Control Optim., \textbf{52} (2014), 2055--2089.

\bibitem{chaves} F. W. Chaves-Silva, X. Zhang, E. Zuazua,
{\it Controllability of evolution equations with memory},
SIAM J. Control Optim., \textbf{55} (2017), 2437--2459.


\bibitem{Davies1995}
E. B. Davies, {\it Spectral theory and differential operators}, Cambridge Studies in Advanced
Mathematics \textbf{42}, Cambridge University Press, Cambridge, 1995.

\bibitem{Ervedoza2008}
S. Ervedoza, {\it Control and stabilization properties for a singular heat equation with an inverse-square potential},
Comm. Partial Differential Equations, \textbf{33} (2008), 1996--2019.

\bibitem{F2016} G. Fragnelli, {\it Interior degenerate/singular parabolic equations in nondivergence form: well-posedness and Carleman estimates}, J. Differential Equations, \textbf{260} (2016), 1314--1371.


\bibitem{FMpress} G. Fragnelli, D. Mugnai, Control of degenerate and singular parabolic equation, in press.
\bibitem{FM2020} G. Fragnelli, D. Mugnai, {\it Singular parabolic equations with interior degeneracy and non smooth coefficients: the Neumann case}, Discrete Contin. Dyn. Syst.-S, \textbf{13} (2020), 1495--1511.
\bibitem{FM2019}  G. Fragnelli, D. Mugnai, {\it Controllability of degenerate and singular parabolic problems: the double strong case with Neumann boundary conditions}, Opuscula Math., \textbf{39} (2019), 207--225.
\bibitem{FM2018} G. Fragnelli, D. Mugnai, {\it Controllability of strongly degenerate parabolic problems with strongly singular
potentials}, Electron. J. Qual. Theory Differ. Equ., \textbf{50} (2018), 1--11.
\bibitem{FM2017} G. Fragnelli, D. Mugnai, {\it Carleman estimates for singular parabolic equations with interior dege\-ne\-racy and non smooth coefficients}, Adv. Nonlinear Anal., \textbf{6} (2017), 61--84.

\bibitem{fu} X. Fu, J. Yong, X. Zhang,
{\it Controllability and observability of a heat equation with hyperbolic memory kernel}, J.Differential Equations, \textbf{247} (2009), 2395--2439.

\bibitem{FI1996}
A. V. Fursikov and O. Y. Imanuvilov,
{\it Controllability of evolution equations},
Lect. Notes Ser. \textbf{34}, Seoul National University,
Seoul, 1996.

\bibitem{lorenzi}
M. Grasselli and A. Lorenzi, {\it Abstract nonlinear Volterra integro-differential equations with nonsmooth kernels}, Atti. Accad. Naz. Lincei Cl. Sci. Fis. Mat. Natur. Rend. Lincei (9) Mat. Appl., \textbf{2} (1991), 43--53

\bibitem{guerrero}
S. Guerrero, O.Yu. Imanuvilov,
{\it Remarks on non controllability of the heat equations with memory},
ESAIM Control Optim. Calc. Var.,  \textbf{19} (2013) 288--300.

\bibitem{hajjaj}
A. Hajjaj,
{\it Estimations de Carleman et Applications \`{a} la contr\^{o}labilit\'e \`{a} Z\'ero D'une Classe
De Syst\`emes Paraboliques D\'eg\'en\'er\'es},
Th\`ese d'Etat, Marrakech, 2013.

\bibitem{halanay}
A. Halanay, L. Pandolfi,
{\it Approximate controllability and lack of controllability to zero of the heat equation with memory},
J.Math. Anal. Appl., \textbf{425} (2015), 194--211.

\bibitem{HP2012} A. Halanay, L. Pandolfi,
{\it Lack of controllability of the heat equation with memory}, Systems and Control Letters, \textbf{61} (2012) 999--1002.

\bibitem{HLP1952}
G. H. Hardy, J. E. Littlewood, G. P\'olya,
{\it Inequalities}, 2nd ed., Cambridge, at the
University Press, 1952.



\bibitem{ivanov}
S. Ivanov, L. Pandolfi,
{\it Heat equation with memory: lack of controllability to rest},
J. Math. Anal. Appl.,  \textbf{355} (2009), 1--11.


\bibitem{lak}
V. Lakshmikantham and M. Rama Mohana Rao,
{\it Theory of Integro-Differential Equations. Stability and Control: Theory, Methods and Applications},
\textbf{1} Gordon and Breach Science Publishers, Lausanne, 1995.

\bibitem{lavanya}
R. Lavanya, K. Balachandran,
{\it Null controllability of nonlinear heat equations with memory effects},
Nonlinear Anal. Hybrid Syst., \textbf{3} (2009), 163--175.



\bibitem{Lions}
J. L. Lions,
{\it Optimal control of systems governed by partial differential equations},
Springer-Verlag, Berlin, 1971.

\bibitem{Lionsb}
J. L. Lions,
{\it Contr\^ole des Syst\`emes Distribu\'es Singuliers},
Gauthier-Villars, Paris, 1983.


\bibitem{MarVan2019} P. Martinez, J. Vancostenoble,
{\it The cost of boundary controllability for a parabolic equation with inverse square potential},
Evol. Equ. Control Theory, \textbf{8} (2019), 397--422.

\bibitem{pandolfi}
L. Pandolfi,
{\it Linear systems with persistent memory: An overview of the biblography on controllability},
arXiv:1804.01865.

\bibitem{pruss}
J. Pr\"{u}ss,
{\it Evolutionary Integral Equations and Applications},
Monographs in Mathematics, \textbf{87} Birkh\"{a}user Verlag, Basel, 1993.


\bibitem{balachandran}
K. Sakthivel, K. Balachandran, B.R. Nagaraj,
{\it On a class of non-linear parabolic control systems with memory effects},
Internat. J. Control, \textbf{81} (2008), 764--777.

\bibitem{simon}
J. Simon,
{\it Compact sets in the space $L^p(0,T;B)$},
Ann. Mat. Pura Appl., \textbf{146} (1986), 65--96.

\bibitem{TG2016}
Q. Tao and H. Gao,
{\it On the null controllability of heat equation with memory},
J. Math. Anal. Appl., \textbf{440} (2016) 1--13.

\bibitem{Vanb} J. Vancostenoble,
{\it Global non-negative approximate controllability of parabolic equations with singular potentials},
In: Alabau-Boussouira F., Ancona F., Porretta A., Sinestrari C. (eds) Trends in Control Theory and Partial Differential Equations.
Springer INdAM Series, \textbf{32} Springer, Cham.

\bibitem{Van2011}
J. Vancostenoble,
{\it Improved Hardy-Poincar\'e inequalities and sharp Carleman estimates for degenerate/singular parabolic problems},
Discrete Contin. Dyn. Syst. Ser. S, \textbf{4} (2011), 761-790.

\bibitem{vz} J. Vancostenoble, E. Zuazua, {\it Hardy inequalities, observability, and control for the wave and Schr\"odinger equations with
singular potentials}, SIAM J. Math. Anal., \textbf{41} (2009),
1508--1532.


\bibitem{VanZua2008} J. Vancostenoble, E. Zuazua,
{\it Null controllability for the heat equation with singular inverse-square potentials}, J. Funct.
Anal., \textbf{254} (2008), 1864--1902.

\bibitem{VazZua2000} J. L. Vazquez and E. Zuazua,
{\it The Hardy inequality and the asymptotic behaviour of the heat equation with an
inverse-square potential}, J. Funct. Anal., \textbf{173} (2000),
103--153.

\bibitem{zhou2016}
X. Zhou and H. Gao,
{\it Controllability of a class of heat equations with memory in one dimension},
Math. Meth. Appl. Sci., \textbf{40} (2017), 3066--3078.

\bibitem{zhou}
X. Zhou and H. Gao,
{\it Interior approximate and null controllability of the heat equation with memory},
Comput. Math. Appl., \textbf{67} (2014), 602--613.



\end{thebibliography}
\end{document}